\documentclass[11pt]{amsart}

\usepackage{amsfonts,epsfig}
\usepackage{latexsym}
\usepackage{amssymb}
\usepackage{amsmath}
\usepackage{amsthm}
\usepackage[pdfencoding=auto,psdextra]{hyperref}
\usepackage{graphics}
\usepackage[table]{xcolor}
\usepackage{verbatim}
\usepackage{enumerate}
\usepackage[all]{xy}
\usepackage{array,booktabs,ctable}
\usepackage{color}
\definecolor{mylinkcolor}{rgb}{0.8,0,0}
\definecolor{myurlcolor}{rgb}{0,0,0.8}
\definecolor{mycitecolor}{rgb}{0,0,0.8}
\hypersetup{colorlinks=true,urlcolor=myurlcolor,citecolor=mycitecolor,linkcolor=mylinkcolor,linktoc=page,breaklinks=true}
\usepackage{bigstrut}

\newtheorem{defn}{Definition}[section]
\newtheorem{definition}[defn]{Definition}

\newtheorem{lemma}[defn]{Lemma}

\newtheorem{thm}[defn]{Theorem}

\newtheorem{cor}[defn]{Corollary}
\newtheorem{prop}[defn]{Proposition}

\theoremstyle{definition}

\newtheorem{remark}[defn]{Remark}
\newtheorem{question}[defn]{Question}
\newtheorem{example}[defn]{Example}

\newcommand{\QQ}{\mathbb Q}

\newcommand{\ZZ}{\mathbb Z}

\newcommand{\Gal}{\operatorname{Gal}}

\newcommand{\GL}{\operatorname{GL}}

\newcommand{\E}{\mathcal{E}}

\newcommand{\tor}{\mathrm{tors}}

\newcommand{\Br}{{\rm Br}}

\begin{document}



\bibliographystyle{plain}
\title[Torsion over $\QQ(D_4^\infty)$ -- Errata]{An Errata for: Torsion subgroups of rational elliptic curves over the compositum of all $D_4$ extensions of the rational numbers}

\author{Harris B. Daniels}
\address{Department of Mathematics and Statistics, Amherst College, Amherst, MA 01002, USA}
\email{hdaniels@amherst.edu}
\urladdr{http://hdaniels.people.amherst.edu}

\keywords{Elliptic Curve, Torsion Points, Galois Theory}

\subjclass[2010]{Primary: 11G05, Secondary: 11R21, 12F10, 14H52.}

\begin{abstract}
In \cite{D4}, the author claims that the fields $\QQ(D_4^\infty)$ defined in the paper and the compositum of all $D_4$ extensions of $\QQ$ coincide. The proof of this claim depends on a misreading of a celebrated result by Shafarevich. The purpose is to salvage the main results of \cite{D4}. That is, the classification of torsion structures of $E$ defined over $\QQ$ when base changed to the compositum of all $D_4$ extensions of $\QQ$ main results of \cite{D4}. All the main results in \cite{D4} are still correct except that we are no longer able to prove that these two fields are equal.
\end{abstract}

\maketitle

\section{Introduction}

The issue that this paper addresses arises at the end of \cite[Section 1]{D4} where the author asserts that the compositum of all $D_4$ extensions of $\QQ$, called $F$ in that paper and $\QQ_{D_4}$ here, and the compositum of all generalized $D_4$ extensions of $\QQ$, see \cite[Definition 1.10]{D4}, denoted $\QQ(D_4^\infty)$, are the same. The author says that this follows from a simple argument as well as a celebrated theorem of Shafarevich. The argument is made that since $D_4$ is nilpotent the kernels of all the embedding problems must also be nilpotent and hence are solvable. The problem is that there is another condition that must be met in order to apply said theorem. For the reader's convenience we include a statement of the theorem by Shafarevich below.

\begin{thm}\cite{Ish, serreTopics, Sha}\label{thm:Sha}
Let $L/K$ be an extension of number fields with Galois group $S$, let $U$ be a nilpotent group with $S$-action, and let $G$ be a semi-direct product $U\rtimes S$. Then the embedding problem for $L/K$ and 
$$
1\rightarrow U \rightarrow G\rightarrow S\rightarrow 1
$$
has a solution.
\end{thm}

As Theorem \ref{thm:Sha} indicates, in order to use this theorem it is not sufficient for the embedding problem to have a nilpotent kernel, it must {\it also} be split. While $D_4$ itself fits into a short exact sequence that is split, it is not the case that all the relevant embedding problems can be made to split. 

\begin{example}
Consider $K = \QQ(\sqrt{d})$. It is a classical result that $K$ can only be embedded into a degree 4 cyclic extension of $\QQ$ if $d$ can be written as the sum of two rational squares. In order words, the embedding problem given by 
$$
1 \rightarrow \ZZ/2\ZZ \rightarrow \ZZ/4\ZZ \rightarrow \ZZ/2\ZZ \rightarrow 1
$$
is not always solvable despite having nilpotent kernel.

In contrast, the problem of embedding a single quadratic extension of $\QQ$ into a $D_4$ extension of $\QQ$ is always possible since the relevant embedding problem
$$
1 \rightarrow \ZZ/4\ZZ \rightarrow D_4\simeq \ZZ/4\ZZ \rtimes \ZZ/2\ZZ \rightarrow \ZZ/2\ZZ \rightarrow 1
$$
is split. Here we should be careful to point out that a Galois extension $L/\QQ$ with Galois group $D_4$ in fact contains 3 different quadratic extensions of $\QQ$. Using the Galois correspondence, we can see that the embedding problem above corresponds to embedding $\QQ(\sqrt{d})$ into $L$ in the only way that makes the extension $L/\QQ(\sqrt{d})$ have Galois group $\ZZ/4\ZZ$. 
\end{example}

For the rest of the article we will let $\QQ_{D_4}$ be the compositum of all $D_4$ extensions of $\QQ$, $\QQ(D_4^\infty)$ be the compositum of all generalized $D_4$ extensions of $\QQ$, and we aim to prove the following theorem:

\begin{thm}\label{thm:goal}
Let $E/\QQ$ be an elliptic curve. Then $E(\QQ_{D_4})_\tor = E(\QQ(D_4^\infty))_\tor$. 
\end{thm}

Once Theorem \ref{thm:goal} is proven all of the results in \cite{D4} are completely justified except the assertion that $\QQ_{D_4} = \QQ(D_4^\infty)$. Since this assertion was not a central part of the paper we don't attempt to prove it here, but we do discuss what is left to show in Section \ref{sec:open}. 

For the ease of the reader, in what follows any groups defined by a small group label will have a link to the corresponding LMFDB page \cite{lmfdb} if it exists and if there is a more commonplace name or description, we use that in place of the small group label.

\subsection{Acknowledgements} We would like to thank Maarten Derickx for pointing this mistake out and for helpful conversations throughout the preparation of this document. We would also like to thank Jackson Morrow and the anonymous referee for helpful comments on an initial draft of this paper. 

\section{Galois Embedding Problems and the Brauer Group of a Field} 

We start this section by introducing Galois embedding problems and exploring a particular problem that illustrates the need for this paper. For the remainder of this section unless otherwise stated we will let $L/K$ be a Galois extension with Galois group $Q$ and 
\begin{equation*}\label{eq:ses}
1 \rightarrow N \rightarrow G \rightarrow Q \rightarrow 1
\end{equation*}
be a short exact sequence of groups. We want to study when there exists a Galois extension $M/K$ with $L\subseteq M$ and $\Gal(M/K)\simeq G$. If such an $M$ exists, we say that $M$ is a solution to the embedding problem given by $(L/K,G,N)$.

One of the main tools that is used to study Galois embedding problems is the Brauer group of a field. 

\begin{defn}
Let $K$ be a field. The {\it Brauer group of $K$}, denoted $\Br(K)$, is the abelian group whose elements are Morita equivalence classes of central simple algebras over $K$ under the binary operation of tensor products. 
\end{defn}

For $a,b\in K$ we let $(a,b)_K$ (or just $(a,b)$ when it is clear from context) represent the class of central simple algebras equivalent to the quaternion algebra $K(i,j)/\langle i^2 = a,j^2 = b, ij = -ji\rangle$ inside $\Br(K)$. These algebras have order 2 in $\Br(K)$ and are deeply connected to the embedding problems that we are interested in.

The main result that we will use to prove Theorem \ref{thm:goal} is known as the embedding criterion and it gives criterion for the existence of a solution to the embedding problem of the form \\$$(F(\sqrt{a_1},\dots,\sqrt{a_r})/F, G, \ZZ/2\ZZ),$$
where $F$ is a field, the $a_i$'s are independent modulo $(F^\times)^2$, and $G$ is a non-split central extension of $(\ZZ/2\ZZ)^r$.

\begin{thm}\label{thm:embedding} {\rm (Embedding Criterion, \cite[Section 7]{Frohlic})}  Let $F$ be a field and $K = F(\sqrt{a_1},\dots,\sqrt{a_r})$ were the $a_i$'s are in $F^\times$ and independent modulo $(F^\times)^2$. Let $Q = \Gal(K/F)$, and consider a non-split central extension $G$ of $\ZZ/2\ZZ$ by $Q$. Let $\sigma_1,\dots,\sigma_r$ generate $Q$, where $\sigma_i(\sqrt{a_j}) = (-1)^{\delta(i,j)}\sqrt{a_j}$ and let $\tilde{\sigma_1},\dots,\tilde{\sigma_r}$ be any set of preimages of $\sigma_1,\dots,\sigma_r$ in $G$. Define $c_{ij}$ for $i\leq j$ by $c_{ij} = 1$ if and only if $[\tilde{\sigma}_i,\tilde{\sigma}_j]\neq 1$ for $i<j$ and $c_{ii} = 1$ if and only if $(\tilde{\sigma}_i)^2 \neq 1$. There exists a Galois Extension $L/F$, $L \supseteq K$, such that $\Gal(L/F)\simeq G$ and the surjection $G \to Q$ is the natural surjection of Galois groups, if and only if 
$$\prod_{i\leq j} (a_i,a_j)^{c_{ij}} = 1 \in \Br(F).$$
\end{thm}

Using Theorem \ref{thm:embedding}, we get the following result concerning polyquadratic extensions of $\QQ$.

\begin{cor} \cite{Embedding, Smith}  \label{cor:embedding} Let $a,b,c,d\in\ZZ$ be nonsquare and independent mod $\ZZ^2$. \\
\begin{enumerate}
\item The field $\QQ(\sqrt{a})$ can be embedded into a $\ZZ/4\ZZ$ extension of $\QQ$ if and only if $$(a,a)=(a,-1) = 1\in \Br(\QQ).$$
\item The field $\QQ(\sqrt{a},\sqrt{b})$ can be embedded into a $D_4$ extension of $\QQ$ if and only if $$(a,b) = 1\in \Br(\QQ).$$
\item The field $\QQ(\sqrt{a},\sqrt{b})$ can be embedded into a $Q_8$ extension of $\QQ$ if and only if $$(a,b)(ab,-1) = 1 \in \Br(\QQ).$$
\item The field $\QQ(\sqrt{a},\sqrt{b},\sqrt{c})$ can be embedded into a $\href{http://www.lmfdb.org/GaloisGroup/16T11}{\rm SmallGroup(16,13)}$ extension of $\QQ$ if and only if $$ (a,ab)(c,-1) = 1 \in \Br(\QQ).$$
\item The field $\QQ(\sqrt{a},\sqrt{b},\sqrt{c},\sqrt{d})$ can be embedding into a $\href{http://www.lmfdb.org/GaloisGroup/32T9}{\rm SmallGroup(32,49)}$ extension of $\QQ$ if and only if $$ (a,b)(c,d)=1 \in \Br(\QQ).$$
\end{enumerate}
\end{cor}

We will use the information in Corollary \ref{cor:embedding} in order to prove that if $K/\QQ$ has Galois group $G$ which is isomorphic to one of the 5 groups in Corollary \ref{cor:embedding}, then $K$ must be in $\QQ_{D_4}$. 

\begin{remark}
While it may not seem like it at first glance, everything in Corollary \ref{cor:embedding} is completely explicit. For example, given a field $K = \QQ(\sqrt{a},\sqrt{b})$ such that $(a,b)= 1\in\Br(\QQ)$, every $D_4$ extension of $\QQ$ that contains $K$ can be given concretely. Since $(a,b)=1 \in \Br(\QQ)$ there exist $\alpha,\beta\in\QQ$ such that $b=\alpha^2-a\beta^2$. Then for any $r\in\QQ^\times$, the extension $\QQ(\sqrt{r(\alpha+\beta\sqrt{a})},\sqrt{\beta})/\QQ$ is a $D_4$ extension that contains $K$ and every such extension arises in this way.
\end{remark}

\section{Proof of the main theorem}

Before outlining the proof of Theorem \ref{thm:goal}, we give a definition and lemma.

\begin{definition}
Given a group $G$, we say that a field $F/\QQ$ is {\bfseries $G$-complete} if for every Galois extension $K/\QQ$ such that $\Gal(K/\QQ) \simeq G$ it follows that $K\subseteq F$. The field $F$ is said to be {\bfseries $G$-incomplete} if there exists at least one Galois extension $K/\QQ$ with $\Gal(K/\QQ)\simeq G$ and $K\not\subseteq F$.
\end{definition}

\begin{lemma}\label{lem:contained}
Let $F$ be an extension of $\QQ$ and let $K/\QQ$ be a finite Galois extension with Galois group $G$. If there exist normal subgroups $N_1,\dots, N_r \lhd G$ such that

\begin{enumerate}
\item for all $1\leq i \leq r$ the field $F$  is $G/N_i$-complete, and 
\item $\displaystyle \bigcap_{i=1}^r N_i = \{e\}$,
\end{enumerate}
then $K\subseteq F$.
\end{lemma}

\begin{proof}
For each $1\leq i \leq r$ let $K_i$ be the subfield of $K$ with Galois group $N_i$ guaranteed to exists by the fundamental theorem of Galois theory. Each of these fields is in $F$ by condition (1) and condition (2) ensures that $K = K_1K_2\dots K_r$. Since each $K_i$ is in $F$ we have that $K = K_1K_2\dots K_r$ is contained in $F$.
\end{proof}

\begin{remark}
Clearly, if $K/\QQ$ is a finite extension that is not Galois with Galois closure $K^{\rm Gal}$ such that $K^{\rm Gal}/\QQ$ satisfies conditions of Lemma \ref{lem:contained} with $F=\QQ_{D_4}$, then $K\subseteq K^{\rm Gal} \subseteq \QQ_{D_4}$.
\end{remark}

With Lemma \ref{lem:contained} in hand, we start to see how one might prove Theorem \ref{thm:goal}. 
The idea is to first show that $\QQ_{D_4}$ is $G$-complete for a few carefully chosen groups $G$ and then use these together with Lemma \ref{lem:contained} to prove that
$$\QQ(E(\QQ(D_4^\infty))[p^\infty])\subseteq \QQ_{D_4}.$$
That is, we will show that for every elliptic curve $E/\QQ$ and prime $p$, the field of definition of the $p$-primary component of $E(\QQ(D_4^\infty))$ is contained in $\QQ_{D_4}$. 
This combined with the fact that $\QQ_{D_4}\subseteq\QQ(D_4^\infty)$ would give us that $E(\QQ_{D_4})[p^\infty] = E(\QQ(D_4^\infty))[p^\infty]$. From there, Theorem \ref{thm:goal} would follow from the classification of finite abelian groups.

We start the process by proving that $\QQ_{D_4}$ is $G$-complete for some small groups.

\begin{prop}\label{prop:Simple_G_in_F}
If $G$ is a group isomorphic to $\ZZ/2\ZZ$, $\ZZ/4\ZZ$, or $D_4$ then $\QQ_{D_4}$ is $G$-complete.
\end{prop}

\begin{proof} Let $K/\QQ$ be a number field with $\Gal(K/\QQ)\simeq G$.
\begin{enumerate}
\item Suppose $G$ is isomorphic to $\ZZ/2\ZZ$. In this case $K=\QQ(\sqrt{a})$ for some squarefree $a$ in $\ZZ$ and we saw from Theorem \ref{thm:Sha} that any quadratic extension of $\QQ$ can be embedded into a $D_4$ extension. Thus $K\subseteq \QQ_{D_4}$ and $\QQ_{D_4}$ is $\ZZ/2\ZZ$-complete.

\item Suppose $G$ is isomorphic to $D_4$. Clearly $\QQ_{D_4}$ is $D_4$-complete.

\item Suppose $G$ is isomorphic to $\ZZ/4\ZZ$. This means that $K$ has a unique subfield of the form $\QQ(\sqrt{a})$. 
This unique quadratic field can be embedded into a $D_4$ extension $L$ such that $L/\QQ(\sqrt{a})$ is not a cyclic degree 4 extension and $K$ is not a subfield of $L$. 
We now consider the Galois extension given by the compositum of $L$ and $K$. 
Since the intersection of $L$ and $K$ is exactly $\QQ(\sqrt{a})$, we have that $\Gal(LK/\QQ)$ is isomorphic to a subgroup of index 2 inside $D_4\times C_4$, call it $G$. 
We also know that if $\pi_1$ and $\pi_2$ are the projection maps out of $D_4\times C_4$, then $\pi_1(G)= D_4$ and $\pi_2(G)= C_4$.
Since $L\cap K= \QQ(\sqrt{a})$, $G$  contains a normal subgroup $N$ such that $[D_4:\pi_1(N)] = [C_4:\pi_2(N)] = 2$ (so the fixed field that is in both $K$ and $L$ is a quadratic) and $\pi_1(N)\neq \langle f \rangle$ (since $\QQ(\sqrt{a})/\QQ$ is {\it not} cyclic). 
If we let $N_1$ and $N_2$ be the 2, noncyclic subgroups of $D_4$ of index 2, we can rephrase these conditions to say that $G$ has the property that either $G \cap \pi_1^{-1}(N_1) = G \cap \pi_2^{-1}(\langle c^2 \rangle)$ or $G \cap \pi_1^{-1}(N_2) = G \cap \pi_2^{-1}(\langle c^2 \rangle)$ and the common group is an index 2 subgroup of $G$. 
Searching the index 2 subgroups of $D_4\times C_4$, we see that there are two groups up to conjugation with this property and they are both isomorphic to $\href{http://www.lmfdb.org/GaloisGroup/16T10}{\rm SmallGroup(16,3)}$.  
With this $\Gal(LK/\QQ)$ is isomorphic to $\href{http://www.lmfdb.org/GaloisGroup/16T10}{\rm SmallGroup(16,3)}$ and letting $\Gal(LK/\QQ) = H$ we see that $H$ has two distinct normal subgroups $N_1$ and $N_2$ such that $N_1\cap N_2 =\{e\}$ and $H/N_1\simeq H/N_2\simeq D_4$. 
Thus by Lemma \ref{lem:contained} and part (2) of this proposition, $K\subseteq KL \subseteq \QQ_{D_4}$ and $\QQ_{D_4}$ is $\ZZ/4\ZZ$-complete

%

\end{enumerate}
\end{proof}

The following lemmas give a list of small groups for which $\QQ_{D_4}$ is not necessarily $G$-complete, but with some added assumptions on $K/\QQ$, we can show that $K\subseteq \QQ_{D_4}$. 

\begin{lemma}\label{lem:Q8_condly}
	Let $G$ be isomorphic to {\rm \href{http://www.lmfdb.org/GaloisGroup/8T5}{SmallGroup(8,4)}} and let $F = \mathbb{Q}(\sqrt{a},\sqrt{b})$ such that 
	$$(a,b)=(ab,-1)=1\in{\rm Br}(\QQ).$$ If $K/\QQ$ is a solution to the embedding problem $(F/\QQ,G,\ZZ/2\ZZ)$, then $K\subseteq \QQ_{D_4}$.
\end{lemma}

\begin{proof}
	From Corollary \ref{cor:embedding} and the added assumption that $(a,b) =1 \in{\rm Br}(\QQ)$, we know that there is a Galois extension $L/\QQ$ such that $\QQ(\sqrt{a},\sqrt{b})\subseteq L$ and $\Gal(L/\QQ)\simeq D_4$. 
	Considering the compositum of $K$ and $L$ we see that $\Gal(LK/\QQ)$ is isomorphic to the uniqe group of order 16 with a quotient isomorphic to $D4$ and a quotient isomorphic to $Q_8$, namely $H \simeq \href{http://www.lmfdb.org/GaloisGroup/16T8}{\rm Smallgroup(16,4)}$. 
	The group $H$ has 3 normal subgroup $N_1$, $N_2$, and $N_3$ such that $N_1\cap N_2\cap N_3=\{e\}$ and $H/N_1\simeq D_4$ and $H/N_2\simeq H/N_3 \simeq \ZZ/4\ZZ$. Therefore, by Lemma \ref{lem:contained} and Proposition \ref{prop:Simple_G_in_F} , we have $K\subseteq LK \subseteq \QQ_{D_4}$.
\end{proof}

\begin{remark}
Here we point out that the previous proof, we didn't actually need all 3 of the normal subgroups. 
In fact, $N_1\cap N_2 = N_1 \cap N_3 = \{ e \}$ and so one in fact only needs two fields tow generate $L$.
We included all 3 in the proof because this is the way the search was conducted, by computing \emph{all} the relevant normal subgroups and computing their intersection.
The code used to verify all of these computations can be found at \cite{MagmaCodeForErrata}.
\end{remark}

\begin{lemma}\label{lem:16-3_condly}
	Let $G$ be isomorphic to {\rm \href{http://www.lmfdb.org/GaloisGroup/8T11}{SmallGroup(16,13)}} and let $F = \mathbb{Q}(\sqrt{a},\sqrt{b},\sqrt{c})$ such that $(a,ab)=(c,-1)=1\in{\rm Br}(\QQ)$. If $K/\QQ$ is a solution to the embedding problem $(F/\QQ,G,\ZZ/2\ZZ)$, then $K\subseteq \QQ_{D_4}$.
\end{lemma}

\begin{proof}
From Corollary \ref{cor:embedding} and the assumption that $(c,-1) = 1\in\Br(\QQ)$ the field $\QQ(\sqrt{c})$ can be embedded into a degree 4 cyclic extension $L/\QQ$.
 Again considering the field $LK$ we compute that $\Gal(LK/\QQ) \simeq D_4\times \ZZ/4\ZZ \simeq \href{http://www.lmfdb.org/GaloisGroup/32T5}{\rm SmallGroup(32,25)}$ and thus we have $K\subseteq LK\subseteq \QQ_{D_4}$.
\end{proof}

\begin{lemma}\label{lem:32-49_condly}
	Let $G$ be isomorphic to {\rm \href{http://www.lmfdb.org/GaloisGroup/8T22}{SmallGroup(32,49)}} and let $F = \mathbb{Q}(\sqrt{a},\sqrt{b},\sqrt{c},\sqrt{d})$ such that $(a,b)=(c,d)=1\in{\rm Br}(\QQ)$. If $K/\QQ$ is a solution to the embedding problem $(F/\QQ,G,\ZZ/2\ZZ)$, then $K\subseteq \QQ_{D_4}$.
\end{lemma}

\begin{proof}
From Corollary \ref{cor:embedding} and the added the assumption that $(a,b) = 1\in \Br(\QQ)$ we have that $\QQ(\sqrt{a},\sqrt{b})$ can be embedded into a $D_4$ extension, we pick one of the possible solutions to this embedding problem and call it $L$. Again we consider the extension $LK/\QQ$ and compute that $\Gal(LK/\QQ) = H\simeq \href{http://www.lmfdb.org/GaloisGroup/32T215}{\rm SmallGroup(64,231)}$. The group $H$ has 3 normal subgroups $N_1, N_2,$ and $N_3,$ such that $N_1\cap N_2\cap N_3 = \{e\}$, $H/N_1\simeq H/N_2\simeq D_4$, and $H/N_3\simeq \href{http://www.lmfdb.org/GaloisGroup/16T11}{\rm SmallGroup(16,13)}$. Thus by Lemma \ref{lem:contained} and Proposition \ref{prop:Simple_G_in_F}, we have $K\subseteq LK \subseteq \QQ_{D_4}$.  

\end{proof}

Before we can prove Theorem \ref{thm:goal} we will need to prove that $\QQ_{D_4}$ is $G$-complete for some groups of size 64. In order to do this we will first need a group theory lemma and a proposition about Galois embedding problems with decomposable kernel. 

\begin{lemma}\label{lem:grps_wedge}\cite[Lemma 4.1]{32}
Let $G$ be a group and let $N_1$, and $N_2$ be normal subgroups of $G$ such that $N_1\cap N_2 = \{e\}$. Then we have that $G$ is isomorphic to the pullback $(G/N_1)\curlywedge(G/N_2)$ and the following diagram commutes:
$$\xymatrix{  
  &  & 1\ar[d] & 1\ar[d] &  \\
  &  & N_2\ar[d]\ar@{=}[r] & N_2\ar[d] &  \\
 1\ar[r] & N_1\ar[r]\ar@{=}[d] & G\ar[r]\ar[d] & G/N_1\ar[r]\ar[d] & 1 \\
 1\ar[r] & N_1\ar[r] & G/N_2\ar[r]\ar[d] & G/N_1N_2\ar[r]\ar[d] & 1 \\
  &  & 1 & 1 &  \\
}$$
Here the maps from a group to a quotient of that group are the natural maps. 
\end{lemma}

Using this lemma, one can prove the following proposition.

\begin{prop}\cite[Theorem 4.1]{32}\label{prop:Central_Product_Embedding}
Let $G$, $N_1$ and $N_2$ be group as in Proposition \ref{lem:grps_wedge}. Let $K/\QQ$ be a Galois extension with $\Gal(K/\QQ)\simeq G/N_1N_2$. Then the embedding problem $(K/\QQ,G,N_1\times N_2)$ is solvable if and only if $(K/\QQ,G/N_1,N_2)$ and $(K/\QQ,G/N_1,N_2)$ are solvable. 
\end{prop}

\begin{remark}
Proposition \ref{prop:Central_Product_Embedding} can be stated more generally for an arbitrary base-field, but we only state it with $\QQ$ as the base field so for the sake of simplicity. For more details, the reader is encouraged to see \cite{32}.
\end{remark}

We are now ready to show that $F$ is $G$-complete with respect to some larger groups.

\begin{prop}\label{prop:64_G_in_F}
Let $G$ be isomorphic to {\rm {SmallGroup(64,206)}, {SmallGroup(64,215)},} or {\rm SmallGroup(64,216)}. Then $\QQ_{D_4}$ is $G$-complete. 
\end{prop}

\begin{proof}
Let $K/\QQ$ be a Galois extension with Galois group $G$, one of the three groups listed above. Then, in all 3 cases $[G,G] \simeq (\ZZ/2\ZZ)^2$ and $G/[G,G]\simeq (\ZZ/2\ZZ)^4.$ Therefore, we can write $[G,G] = N_1\times N_2$ where $G$, $N_1$, and $N_2$ satisfy the conditions of Lemma \ref{lem:grps_wedge}. Choosing $N_1$ and $N_2$ correctly we get that in all cases $G/N_1\simeq D_4\times (\ZZ/2\ZZ)^2$ and 
$$G/N_2\simeq 
\begin{cases}
{\rm SmallGroup(32,48)}  & G\simeq {\rm SmallGroup(64, 206)}\\ 
{\rm SmallGroup(32,49)}  & G\simeq {\rm SmallGroup(64, 215)\ or\ SmallGroup(64, 216).} \\
\end{cases}$$
The group SmallGroup(32,48) is isomorphic to ${\rm SmallGroup(16,13)}\times \ZZ/2\ZZ$. Using Proposition \ref{prop:Central_Product_Embedding}, we get that the field $\QQ(\sqrt{a},\sqrt{b},\sqrt{c},\sqrt{d})$ can be embedded into a field with Galois group $G$ if and only if 
$$\begin{cases}
(a,b) = (a,b)(c,-1) = 1 \in\Br(\QQ) & G\simeq {\rm SmallGroup(64, 206)}\\ 
(a,b) = (a,b)(c,d) = 1 \in\Br(\QQ) & G\simeq {\rm SmallGroup(64, 215)\ or\ SmallGroup(64, 216).} \\
\end{cases}$$
Rewriting these conditions we get
$$\begin{cases}
(a,b) = (c,-1) = 1 \in\Br(\QQ) & G\simeq {\rm SmallGroup(64, 206)}\\ 
(a,b) = (c,d) = 1 \in\Br(\QQ) & G\simeq {\rm SmallGroup(64, 215)\ or\ SmallGroup(64, 216).} \\
\end{cases}$$
Thus, any Galois extension $K/\QQ$ with $\Gal(K/\QQ)\simeq G$ is the compositum of two fields of degree 32 that are in $\QQ_{D_4}$ by Lemmas \ref{lem:Q8_condly}, \ref{lem:16-3_condly}, and \ref{lem:32-49_condly}. Therefore it must be that $K\subseteq \QQ_{D_4}$ and $\QQ_{D_4}$ is $G$-complete for all 3 groups. 

\end{proof}

With this we are now ready to prove Theorem \ref{thm:goal} but before we get started we make a remark that will guide our strategy. 

\begin{remark}\label{rmk:lift}
Examining the proofs of Propositions \ref{prop:Simple_G_in_F}, and \ref{prop:64_G_in_F} as well as Lemmas \ref{lem:Q8_condly}, \ref{lem:16-3_condly}, and \ref{lem:32-49_condly}, we see that at each step it was useful to 
think of the field in question as living in a larger but still finite extension of $\QQ$. 
For example in the proof of Proposition \ref{prop:Simple_G_in_F} while proving that $\QQ_{D_4}$ is $\ZZ/4\ZZ$-complete, we start with a generic $\ZZ/4\ZZ$ extension of $\QQ$, called $K/\QQ$ in the proof. The result follows from considering $K$ inside of the field $LK$, where $L/\QQ$ is a $D_4$ extension and $[L\cap K : \QQ] = 2.$ 

Motivated by this, when we are trying to compute the $p$-primary component of $E(\QQ_{D_4})$ and show that it is 
$\ZZ/p^n\ZZ\oplus\ZZ/p^m\ZZ$ with $m\leq n$ using this method, it is useful to not just consider the field of definition of the $p^n$-th torsion but actually the field of definition of the $p^{n+1}$ torsion. 
On the groups side of things, this just means that we will lift groups of level $p^n$ to level $p^{n+1}$ before doing any of the computations. 
The idea is that viewing $\QQ(E[p^n])$ inside $\QQ(E[p^{n+1}])$, allows us to see that the relevant fields are generated by fields that are in $\QQ_{D_4}$.

\end{remark}

\begin{proof}[Proof of Theorem \ref{thm:goal}]
Let $E/\QQ$ be an elliptic curve and $p$ a prime. We denote the $p$-primary component of $E(\QQ(D_4^\infty))$ by $E(\QQ(D_4^\infty))[p^\infty]$ and its field of definition by $K_p = \QQ(E(\QQ(D_4^\infty)[p^\infty])).$
 
The overall strategy of the proof is to show that $K_p\subseteq \QQ_{D_4}$. 
This combined with the fact that $\QQ_{D_4}\subseteq\QQ(D_4^\infty)$ would yield that $E(\QQ(D_4^\infty))[p^\infty] = E(\QQ_{D_4})[p^\infty]$ and from here the result follows from the classification of finite abelian groups. We proceed by breaking the argument down into cases based on the size and parity of $p$.

Suppose $p\geq 5$ is a prime. From \cite[Sections 5.1-3]{D4}, we know that in this case $K_p/\QQ$ is an abelian extension. Since $\Gal(K_p/\QQ)$ has exponent dividing 4 and is abelian, it must be isomorphic to $(\ZZ/4\ZZ)^{s_1} \times (\ZZ/2\ZZ)^{s_2}$ for some nonnegative integers $s_1$ and $s_2$, and since $\QQ_{D_4}$ is $\ZZ/2\ZZ$- and $\ZZ/4\ZZ$-complete $K_p \subseteq \QQ_{D_4}$ and $E(\QQ(D_4^\infty))[p^\infty] = E(\QQ_{D_4})[p^\infty]$. 

Next, suppose $p=3$. From \cite[Section 5.4]{D4}, we see that there are 3 different cases that need to be checked corresponding to the possible isomorphism class of $\E(\QQ(D_4^\infty))[3^\infty]$. For two of them $\Gal(K_p/\QQ)\simeq \ZZ/2\ZZ$ and in this case clearly $K_3\subseteq\QQ_{D_4}$. 
The last cases corresponds to the case when $\Gal(K_3/\QQ)$ is isomorphic to a subgroup of the normalizer of the split Cartan subgroup of $\GL_2(\ZZ/3\ZZ)$. 
It turns out that the normalizer of the non-split Cartan subgroup of $\GL_2(\ZZ/3\ZZ)$ is isomorphic to $D_4$ and again we have $K_3\subseteq \QQ_{D_4}$.

The last outstanding case is when $p=2$. To settle this case, let $\mathcal{S}$ be the set containing of isomorphism classes of groups with small groups label  (2,1),(4,1),(8,3),(64,206), (64,215), or (64,216). 
When a we say a group is in $\mathcal{S}$, we mean that it is isomorphic to a group in $\mathcal{S}$. 

Next for each of the relevant groups from \cite{RZB} and as discussed in Remark \ref{rmk:lift}, consider them at twice the level necessary to verify \cite[Table 2]{D4}.  
For each of these groups $G$, we let 
$$S_G = \{N : N\triangleleft G \hbox{ and $G/N\in\mathcal{S}$}\} \hbox{ and }N_G = \displaystyle \bigcap_{N\in S_G} N.$$ 

Now, if $E/\QQ$ is comes from a point on the modular curve corresponding to one of the groups $G$ in \cite[Table 2]{D4}, then there is a subfield of $\QQ(E[2^n])$ call it $L_G$ corresponding to the fixed field of $N_G$ and with Galois groups $G/N_G$. 
The field $L_G$ is contained in $\QQ_{D_4}$ since $\QQ_{D_4}$ is $G$ complete for all $G\in \mathcal{S}$ by Propositions \ref{prop:Simple_G_in_F} and \ref{prop:64_G_in_F}. 
Further we know that $E(L_G)[2^\infty]$ is isomorphic to the subgroup of $(\ZZ/2^n\ZZ)^2$ that is fixed by $N_G$ where $2^n$ is the level of $G$.
So if we can confirm that $E(L_G)[2^\infty]\simeq E(\QQ(D_4^\infty))[2^\infty]$, then we would have $E(\QQ_{D_4})[2^\infty] \simeq E(\QQ(D_4^\infty))[2^\infty]$ since $E(L_G)\subseteq E(\QQ_{D_4})\subseteq E(\QQ(D_4^\infty))$. 
All that is left to do is show that for each $G$ of interest the fixed group of $N_G$ is isomorphic to the subgroup given in \cite[Table 2]{D4}. 
Thus $E(D_4^\infty)[2^\infty] = E(K_G)[2^\infty] \subseteq E(\QQ_{D_4})[2^\infty] \subseteq E(D_4^\infty)[2^\infty]$ implying that $E(\QQ_{D_4})[2^\infty] = E(\QQ_{D_4}^\infty)[2^\infty]$. 
Thus completing the case when $p=2$ as well as the proof of Theorem \ref{thm:goal}. The code confirming this can be found in \cite{MagmaCodeForErrata}.\end{proof}

\section{Remaining Open Questions}\label{sec:open}

The question about the equality of the fields $\QQ(D_4^\infty)$ and $\QQ_{D_4}$ still remains open, but there are some smaller questions one could ask that could potentially shed light on the relationship between these two fields. 

\begin{question}
Is $\QQ_{D_4}$ a $Q_8$-complete field?
\end{question}

We saw that under the added assumption that each individual term coming from the embedding criterion is trivial we can show that a $Q_8$ extension of $\QQ$ is in $\QQ_{D_4}$, but the question remains in the case when $(a,b) = (ab,-1) \neq 1 \in \Br(\QQ)$. 
More concretely, consider the field $\QQ(\alpha)/\QQ$ where $\alpha$ is a root of 
$$f(x) = x^8 + 84x^6 + 1260x^4 + 5292x^2 + 441.$$
In this case, we have that $\Gal(K/\QQ)\simeq Q_8$ and $\QQ(\sqrt{3},\sqrt{14}) \subset K$ and $(3,14) = (42,-1) \neq 1 \in \Br(\QQ)$. 
Is $K\subseteq\QQ_{D_4}$? We suspect that if one could find the answer to this particular questions, they could also likely settle the relationship between $\QQ_{D_4}$ and $\QQ(D_4^\infty)$.

\end{document}